\documentclass[10pt]{amsart}
\usepackage{amssymb}
\usepackage{bm}
\usepackage{graphicx}
\usepackage[centertags]{amsmath}
\usepackage{amsfonts}
\usepackage{amsthm}
\usepackage{xcolor}
\newtheorem{thm}{Theorem}
\newtheorem{cor}[thm]{Corollary}
\newtheorem{lem}[thm]{Lemma}

\newtheorem{prop}[thm]{Proposition}
\newtheorem{claim}[thm]{Claim}
\newtheorem{fact}[thm]{Fact}
\newtheorem{defn}[thm]{Definition}
\theoremstyle{definition}
\newtheorem{rem}[thm]{Remark}

\newtheorem{notation}{Notation}

\newcommand{\rr}{\mathbb{R}}
\newcommand{\nn}{\mathbb{N}}

\newcommand{\qq}{\mathbb{Q}}
\newcommand{\ee}{\varepsilon}

\newcommand{\inn}{\in \mathbb{N}}

\newcommand{\con}{\smallfrown}

\newcommand{\aaa}{\mathcal{A}}
\newcommand{\bbb}{\mathcal{B}}

\newcommand{\ddd}{\mathcal{D}}
\newcommand{\ccc}{\mathcal{C}}
\newcommand{\fff}{\mathcal{F}}

\newcommand{\llll}{\mathcal{L}}

\newcommand{\www}{\mathcal{W}}
\newcommand{\zzz}{\mathcal{Z}}

\newcommand{\sspan}{\mathrm{span}}

\begin{document}

\title{Uniformly factoring weakly compact operators}

\author{Kevin Beanland}
\address{Department of Mathematics and Applied Mathematics, Virginia Commonwealth
University, Richmond, VA 23284.}
\email{kbeanland@vcu.edu}

\author{Daniel Freeman}
\address{Department of Mathematics and Computer Science\\
Saint Louis University , St Louis, MO 63103  USA}
\email{dfreema7@slu.edu}

\thanks{The first author acknowledges support from the Fulbright Foundation - Greece and the Fulbright
program.}
\thanks{The second author is supported by NSF Award 1139143}

\maketitle


\begin{abstract}
Let $X$ and $Y$ be separable Banach spaces. Suppose $Y$ either has a
shrinking basis or $Y$ is isomorphic to $C(2^\nn)$ and $\aaa$ is a
subset of weakly compact operators from $X$ to $Y$ which is analytic
in the strong operator topology. We prove that there is a reflexive
space with a basis $Z$ such that every $T \in \aaa$ factors through
$Z$.  Likewise, we prove that if $\aaa \subset \llll(X, C(2^\nn))$
is a set of operators whose adjoints have separable range
and is analytic in the strong operator topology then there is a
Banach space $Z$ with separable dual such that every
$T \in \aaa$ factors through $Z$. Finally we prove a uniformly version of this result in
which we allow the domain and range spaces to vary.
\end{abstract}

\section{Introduction}

Recall that if $X$ and $Y$ are Banach spaces then a bounded operator
$T:X\rightarrow Y$ is called {\em weakly compact} if
$\overline{T(B_X)}$ is weakly compact, where $B_X$ is the unit ball
of $X$.  If there exists a reflexive Banach space $Z$ and bounded
operators $T_1:X \to Z$ and $T_2:Z \to Y$ with $T=T_2 \circ T_1$
then $T_1$ and $T_2$ are both weakly compact by Alaoglu's theorem
and hence $T:X\rightarrow Y$ is weakly compact as well.  Thus it is
immediate that any bounded operator which factors through a
reflexive Banach space is weakly compact. In their seminal 1974
paper \cite{DFJP}, Davis, Figiel, Johnson and Pe{\l}czy{\'n}ski
proved that the converse is true as well.  That is, every weakly compact
operator factors through a reflexive Banach space. Likewise, every
bounded operator whose adjoint has separable range factors through
a Banach space with separable dual. Using the DFJP
interpolation technique, in 1988 Zippin proved that every separable
reflexive Banach space embeds into a reflexive Banach space with a
basis and that every Banach space with separable dual embeds into a
Banach space with a shrinking basis \cite{Z}.

For each separable reflexive Banach space $X$ we may choose a
reflexive Banach $Z$ with a basis such that $X$ embeds into $Z$.  It
is natural to consider when the choice of $Z$ can be done uniformly.
That is, given a set of separable reflexive Banach spaces $\aaa$,
when does there exist a reflexive Banach space $Z$ with a basis such
that $X$ embeds into $Z$ for every $X\in \aaa$?  Szlenk proved that
there does not exist a Banach space $Z$ with separable dual such
that every separable reflexive Banach space embeds into $Z$
\cite{Sz}. Bourgain proved further that if $Z$ is a separable Banach
space such that every separable reflexive Banach space embeds into
$Z$ then every separable Banach space embeds into $Z$ \cite{Bo}.
Thus, any uniform embedding theorem must consider strict subsets of
the set of separable reflexive Banach spaces. In his Phd thesis,
Bossard developed a framework for studying sets of Banach spaces
using descriptive set theory \cite{Bos,BosPHD}. In this context, it
was shown in \cite{DodosFerenczi} and \cite{OSZ1} that if $\aaa$ is
an analytic set of separable reflexive Banach spaces then there
exists a separable reflexive Banach space $Z$ such that $X$ embeds
into $Z$ for all $X\in \aaa$, and in \cite{DodosFerenczi} and
\cite{FOSZ} it was shown that if $A$ is an analytic set of Banach
spaces with separable dual then there exists a  Banach space $Z$
with separable dual such that $X$ embeds into $Z$ for all $X\in
\aaa$. In particular, solving an open problem posed by Bourgain
\cite{Bo}, there exists a separable reflexive Banach space $Z$ such
that every separable uniformly convex Banach space embeds into $Z$
\cite{OSuconvex}.
 As the set of all Banach spaces which embed into a fixed
Banach space is analytic in the Bossard framework, these uniform
embedding theorems are optimal.

The goal for this paper is to return to the original operator
factorization problem with the same uniform perspective that was
applied to the embedding problems.  That is, given separable Banach
spaces $X$ and $Y$ and a set of weakly compact operators
$\aaa\subset\llll(X,Y)$, we want to know when does there exist a
reflexive Banach space $Z$ such that $T$ factors through $Z$ for all
$T\in\aaa$.  We are able to answer this question in the following
cases.

\begin{thm}
Let $X$ and $Y$ be separable Banach spaces and let $\aaa$
be a set of weakly compact operators from $X$ to $Y$ which is analytic
in the strong operator topology. Suppose either $Y$ has a
shrinking basis or $Y$ is isomorphic to $C(2^\nn)$. Then there is a reflexive Banach space $Z$
with a basis such that every $T \in \aaa$ factors through $Z$.
\label{maintheorem}
\end{thm}

\begin{thm}
Let $X$ be a separable Banach space and let $\aaa \subset \llll(X,
C(2^\nn))$ be a set of bounded operators whose adjoints have
separable range which is analytic in the strong operator topology.
Then there is a Banach space $Z$ with a shrinking basis such that
every $T \in \aaa$ factors through $Z$. \label{maintheoremSD}
\end{thm}

The idea of factoring all operators in a set through a single Banach
space has been considered previously for compact operators and
compact sets of weakly compact operators \cite{ALRR,GoGu,MO}.  In
particular, Johnson constructed a reflexive Banach space $Z_K$ such
that if $X$ and $Y$ are Banach spaces and either $X^*$ or $Y$ has
the approximation property then every compact operator
$T:X\rightarrow Y$ factors through $Z_K$ \cite{JohnsonCompact}.
Later, Figiel showed that if $X$ and $Y$ are Banach spaces and
$T:X\rightarrow Y$ is a compact operator, then $T$ factors through a
subspace of $Z_K$ \cite{Fi}. It is particularly interesting that the
space $Z_K$ is independent of the Banach spaces $X$ and $Y$. In
\cite{BrookerAsplund},  Brooker proves that for every countable
ordinal $\alpha$, if $X$ and $Y$ are separable Banach spaces and
$T:X\rightarrow Y$ is a bounded operator with Szlenk index at most
$\omega^\alpha$   then $T$ factors through a  Banach space with
separable dual and Szlenk index at most $\omega^{\alpha+1}$.  This
result, combined with the embedding result in \cite{FOSZ}  gives
that for every countable ordinal $\alpha$, there exists a Banach
space $Z$ with a shrinking basis such that every bounded operator
with Szlenk index at most $\omega^\alpha$ factors through a subspace
of $Z$.  In section 4 we present generalizations of Theorems
\ref{maintheorem} and \ref{maintheoremSD} where the Banach space $X$
is allowed to vary.

The authors thank Pandelis Dodos for his suggestions and helpful
ideas about the paper. Much of this work was conducted at the National Technical University
of Athens in Greece during the spring of 2011. The first author would like
to thank Spiros Argyros for his hospitality and for providing an excellent research environment during this period.

\section{Preliminaries}

A topological space $P$ is called a {\em Polish} space if it is
separable and completely metrizable. A set $X$, together with a
$\sigma$-algebra $\Sigma$, is called a {\em standard Borel space} if
the measurable space $(X, \Sigma)$ is Borel isomorphic to a Polish
space. A subset $A\subset X$ is said to be {\em analytic} if there
exists a Polish space $P$ and a Borel map $f: P \to X$ with
$f(P)=A$. A subset of $X$ is said to be {\em coanalytic} if its
complement is analytic.

Given some Polish space $X$, we will be studying sets of closed
subspaces of $X$.  Thus, from a descriptive set theory point of
view, it is natural to assign a $\sigma$-algebra to the set of
closed subsets of $X$ which then forms a standard Borel space. Let
$F(X)$ denote the set of closed subspaces of $X$.  The {\em
Effros-Borel $\sigma$-algebra}, $E(X)$, is defined as the collection
of sets with the following generator
$$\big\{ \{F \in F(X) : F \cap U \not= \emptyset\}:U\subset X\textrm{ is open}\big\}.$$
The measurable space $(F(X),E(X))$ is a standard Borel space.  If
$X$ is a Banach space, then $Subs(X)$ denotes the standard Borel
space consisting of the closed subspaces of $X$ endowed with the
relative Effros-Borel $\sigma$-algebra.  As every separable Banach
space is isometric to a subspace of $C(2^\nn)$, the standard Borel
space $SB=Subs(C(2^\nn))$ is of particular importance when studying
sets of separable Banach spaces. 

If $X$ and $Y$ are separable Banach spaces, then the space
$\llll(X,Y)$  of all bounded linear operators from $X$ to $Y$
carries a natural structure as a standard Borel space whose Borel
sets coincide with the Borel sets generated by the strong operator
topology (i.e. the topology of pointwise convergence on nets). In
this paper when we refer to a Borel subset of $\llll(X,Y)$ it is
understood that this is with respect to the Borel $\sigma$-algebra
generated by the strong operator topology. There are several papers
in which $\llll(X,Y)$ is considered with this structure
\cite{BeanlandIsrael,BeanlandDodos,BF-ordinal}.

Both the set of all separable reflexive Banach spaces and the set of
all Banach spaces with separable dual are coanalytic subsets of SB.
This fact is essential in the proofs of the universal embedding
theorems for analytic sets of separable reflexive Banach spaces and
analytic sets of Banach spaces with separable dual
\cite{DodosFerenczi},\cite{FOSZ},\cite{OSZ1}.  Thus, we will
naturally need the following theorem to prove our universal
factorization results for analytic sets of weakly compact operators
and analytic sets of operators whose adjoints have separable range.

\begin{prop}
For $X,Y \in SB$ the following are coanalytic subsets of
$\llll(X,Y)$.
\begin{itemize}
\item[(a)] The set of weakly compact operators.
\item[(b)] The set of operators whose adjoints have separable range (these operators are called Asplund operators).
\end{itemize}
\label{opscoanalytic}
\end{prop}

Before proving Proposition \ref{opscoanalytic}, we will need to
introduce some more results from descriptive set theory. Given a
Polish space $E$, let $K(E)$ be the space of all compact subset of
$E$. The space $K(E)$ is Polish when equipped with the {\em
Vietoris} topology, which is the topology on $K(E)$ generated by the
sets
$$\big\{ \{K \in K(X) : K \cap U \not= \emptyset\}:U\subset X\textrm{ is open}\big\}\text{ and }
\big\{ \{K \in K(X) : K \subseteq U\}:U\subset X\textrm{ is
open}\big\}.$$
 When studying sequences in the unit
ball of a separable Banach space $X$, we note that the the space
$B_X^\nn$ is a Polish space when endowed with the product topology.
We will always consider $B_{\ell_\infty}$, the ball of
$\ell_\infty$, to be equipped with the weak$^*$ topology.
 In \cite{BosPHD} Bossard proves the following Theorem.
\begin{thm}[\cite{BosPHD}]
The set
$$\Sigma= \{K \in K(B_{\ell_\infty}) : K \mbox{ is norm-separable}\}$$
is coanalytic in the Vietoris topology of $K(B_{\ell_\infty})$.
\label{Sigma}
\end{thm}
We will show that for all $X,Y\in SB$ the set of operators whose
adjoints have separable range is coanalytic in $\llll(X,Y)$ by
showing that the set is Borel reducible to $\Sigma\subset
K(B_{\ell_\infty})$.  To do this, we will define a map $\Phi  :
\llll(X,Y) \to K(B_{\ell_\infty})$, and use the following theorem to
show that it is Borel.
\begin{thm}\cite[Theorem 28.8]{Ke}
Let $X$ and $Y$ be Polish spaces and $A \subset Y \times X$ be such
that for each $y \in Y$ the set $A_y =\{x \in X: (y,x) \in A\}$ is
compact. Consider the map $\Phi_A:Y \to K(X)$ defined by $\Phi_A(y)=
A_y$. Then $A$ is Borel if and only if $\Phi_A$ is a Borel map.
\label{slice}
\end{thm}

 By the Kuratowski and
Ryll-Nardzewski selection theorem \cite{KRN} we can find a sequence
of Borel maps $(s_n)_{n \in \nn}$ such that $s_n:F(C(2^\nn)) \to
C(2^\nn)$ for each $n \in \nn$ and $(s_n(E))_{n=1}^\infty$ is dense
in $E$. In addition, for all $n\in\nn$ let  $d_n: SB \to C(2^\nn)$
be a Borel map such that $(d_n(X))_{n \in \nn}$ is dense in $B_X$
for all $X\in SB$ and for $p,q \in \qq$ and $m,k \in \nn$ if
$qd_m(X) + p d_k(X) \in B_X$ then there is an $\ell \in \nn$ with
$d_\ell(X) =qd_m(X) + p d_k(X)$. We will also assume that $d_n(X)
\not=0$ for all $X \in SB$ and $n \in \nn$.  Working with the
sequences $(s_n)$ and $(d_n)$ will be easier for us than dealing
with the Efros-Borel $\sigma$-algebra or Vietoris topology directly.


\begin{proof}[Proof of Proposition \ref{opscoanalytic}]
Item (a) is proved in \cite[Proposition 9]{BF-ordinal} and follows
from the fact that weakly compact operators are exactly those
operators that take bounded sequences in $X$ to sequences that do
not dominate the summing basis of $c_0$.

The proof of (b) requires a bit more effort, but it follows the same outline as the proof that the collection
of all spaces with separable dual (SD) is coanalytic \cite{Bos}.  Let $\aaa(X,Y)$ denote the collection of
operators in $\llll(X,Y)$ whose adjoints have separable range. For $T \in \aaa(X,Y)$ and $y^* \in B_{Y^*}$ Let
$$f_{T^*y^*} = \bigg( \frac{T^*y^*(d_i(X))}{\|d_i(X)\|} \bigg)_{i=1}^\infty \in B_{\ell_\infty}.$$
For $T \in \llll(X,Y)$ let $K_T = \{ f_{T^*y^*} : y^* \in B_{Y^*}\}$. Notice that $K_T$ can be
identified with $T^*(B_{Y^*})$ via the homeomorphism $T^*(y^*) \mapsto f_{T^*y^*}$. Here $T^*(B_{Y^*})$ is endowed with the weak$^*$ topology.
So, $K_T$ is compact in $B_{\ell_\infty}$ with the weak$^*$ topology. Define $D_\llll \subset \llll(X,Y) \times B_{\ell_\infty}$ as
follows
$$(T,f) \in D_\llll \iff f \in K_T.$$
Using the following characterization, the set $D_\llll$ is Borel.
\begin{equation*}
\begin{split}
(T,f) \in D_\llll \iff & \forall n, m, k \in \nn~\forall q, p\in \qq \mbox{ we have}\\
& ( p Td_n(X) +  q Td_m(X)=  Td_k(X) \implies \\
& p\|d_n(X)\| f(n) + q \|d_m(X)\| f(m) = \|d_k(X)\| f(k)).
\end{split}
\end{equation*}
Notice that for each $T \in \llll(X,Y)$ the set $D_T= \{f : (T,f)\in D_\llll\}$ is equal to $K_T$ and is therefore compact. Applying Theorem \ref{slice}, $\Phi  : \llll(X,Y) \to K(B_{\ell_\infty})$ defined by $\Phi(T) = K_T$ is a Borel map. Finally, note that
$$T \in \aaa(X,Y) \iff  \Phi(T) = K_T \in \Sigma= \{K \in K(B_{\ell_\infty}) : K \mbox{ is norm-separable}\}.$$
Using Proposition \ref{Sigma} we have that $\aaa(X,Y)$ is  Borel
reducible to a coanalytic set and is hence itself coanalytic. 
\end{proof}



Concerning reflexive Banach spaces with bases as well as Banach
spaces with bases and separable dual, Argyros and Dodos \cite{ADo}
proved the following deep theorem.

\begin{thm}[\cite{ADo}]
Let $\aaa \subset SB$ be an analytic collection of reflexive Banach
spaces (resp. Banach spaces with separable dual) such that each $X
\in \aaa$ has a basis.  Then there is a reflexive Banach space $Z_\aaa$ (resp. Banach space with separable
dual) with a basis that contains every $X\in \aaa$ as
a complemented subspace. \label{ADcomp}
\end{thm}

Although it is possible for us to apply Theorem \ref{ADcomp} as a
black box, we give some brief description here about how the space
$Z_\aaa$ is constructed.  Let $\aaa \subset SB$ be an analytic
collection of reflexive spaces.  Since the map from $S_{C(2^\nn)}^\nn$
to $SB$ given by $(x_n)_{n\in\nn}\mapsto [x_n]_{n\in\nn}$ is Borel
and the set of basic sequences in a Banach space is Borel, we
obtain an analytic set $\bbb$ of basic sequences in $S_{C(2^\nn)}$
such that for every reflexive Banach space $X\in \aaa$ there exists
$(x_n)\in\bbb$ such that $(x_n)$ is a basis for $X$ and for every
$(x_n)\in\bbb$ we have that $[x_n]\in\aaa$.  Instead of working with
an analytic collection of Banach spaces $\aaa$, we can now work with
an analytic collection of basic sequences $\bbb$.  Argyros and
Dodos, then give a procedure to amalgamate $\bbb$ into a tree basis
$(x_\alpha)_{\alpha\in Tr}$, where $Tr$ is a finitely branching
tree. That is, they construct $(x_\alpha)_{\alpha\in Tr}\subset
S_{C(2^\nn)}$ such that $(x_\alpha)_{\alpha\in Tr}\subset S_{C(2^\nn)}$
is a basic sequence under any ordering which preserves the tree
order, $[x_\alpha]$ is reflexive, and every $(x_n)\in \bbb$ is
equivalent to a branch of $(x_\alpha)_{\alpha\in Tr}$.  Furthermore,
if $(\alpha_n)_{n\in\nn}$ is a branch of $Tr$ then the restriction
operator $P:[x_\alpha]_{\alpha\in
Tr}\rightarrow[x_{\alpha_n}]_{n\in\nn}$ given by $P(\sum a_\alpha
x_\alpha)=\sum a_{\alpha_n} x_{\alpha_n}$ is a bounded projection.
Thus, Theorem \ref{ADcomp} follows from the following results.

\begin{thm}[\cite{ADo}]
Let $\aaa \subset {C(2^\nn)}^\nn$ be an analytic collection of
normalized shrinking and boundedly complete basic sequences. There
is a reflexive Banach space $Z$ with a basis $(z_n)$ such that if
$(x_n)\in \aaa$ then there exists a subsequence $(k_n)$ of $\nn$
such that $(x_n)$ is equivalent to $(z_{k_n})$ and $[z_{k_n}]$ is
complemented in $Z$. \label{ADseq1}
\end{thm}

\begin{thm}[\cite{ADo}]
Let $\aaa \subset {C(2^\nn)}^\nn$ be an analytic collection of
normalized shrinking basic sequences. There is a Banach space $Z$
with a shrinking basis $(z_n)$ such that if $(x_n)\in \aaa$ then
there exists a subsequence $(k_n)$ of $\nn$ such that $(x_n)$ is
equivalent to $(z_{k_n})$ and $[z_{k_n}]$ is complemented in $Z$.
\label{ADseq2}
\end{thm}

Given an analytic collection $\aaa$ of weakly compact operators, our
goal is to obtain an analytic collection $\bbb$ of normalized
shrinking and boundedly complete basic sequences such that for every
$T\in \aaa$, there exists $(x_n)\in \bbb$ such that $T$ factors
through $[x_n]$.  We then are able to apply Theorem \ref{ADseq1} and
obtain a separable reflexive Banach space $Z$ such that every
$T\in\aaa$ factors through a complemented subspace of $Z$.  Hence,
every $T\in\aaa$ factors through $Z$ itself.  This idea of creating
a complementably universal Banach space $Z$ inorder to lift
operators defined on an analytic collection of Banach spaces with
bases was used by Dodos in \cite{DodosQ}, where he characterizes for
what sets of separable Banach spaces $\ccc$ does there exist a
separable Banach space $Z$ such that $\ell_1$ does not embed into
$Z$ and every $X\in\ccc$ is a quotient of $Z$.

\section{Parametrized Factorization}

\begin{notation}
In the rest of the paper we set the following notation.

\begin{itemize}
\item[(a)] $X$ denotes a separable Banach space and $Y$ denotes a Banach space with a Schauder basis.
\item[(b)] Let $T \in \llll(X,Y)$. Denote by $(y_n^T)_{n \in \nn}$ a basis of $Y$ that depends on $T$ and for $k \inn$, let
$P^T_k:Y \to [ y_n^T : n\leqslant k]$ be the natural projection.
\item[(c)] Let $y^T_0=\sum_{n \inn} \frac{1}{2^n} y^T_n$ and $E_T :=\overline{co(T(B_X) \cup\{y^T_0\})}$.

\item[(d)] Define
\begin{equation*}
W_{T} = \overline{\bigcup_{k \inn} P^T_k (E_T)}.
\end{equation*}
Note that $W_T$ is closed, bounded, convex and symmetric. Also, $P^T_k (W_T) \subset W_T $ for each $k \in \nn$.
\item[(e)] Let $W \subset Y$ be closed, convex, bounded and symmetric and for each $m \in \nn$ define
\begin{equation*}
W^m:=2^mW + 2^{-m}B_Y.
\end{equation*}
\item[(f)] Let $\|\cdot\|_{W^m}$ denote the Minkowski gauge norm of the set
$W^m$.  That is,
\begin{equation*}
\|y\|_{W^m} = \inf \{ \lambda >0 :\frac{y}{\lambda} \in W^m\}.
\end{equation*}
\item[(g)] Let
\begin{equation*}
Z_{T}  = \{ z \in Y : \sum_{m=1}^\infty \|z \|_{W^m_{T}}^2 < \infty\} \mbox{ and } \|z\|_T =  \bigg(\sum_{m=1}^\infty \|z \|_{W^m_{T}}^2\bigg)^\frac{1}{2}.
\end{equation*}
\end{itemize}
\end{notation}

The following items are proved in \cite{DFJP}. The reader may also want to consult \cite[Appendix B]{DodosBook} for a nice treatment of this material.
\begin{thm}[\cite{DFJP}]
The following hold.
\begin{itemize}
\item[(a)] There exist $T_1:X \to Z_T$ and $T_2:Z_T \to Y$ such that $T=T_2 T_1$; in other words, $T$ factors through
$Z_T$.  Furthermore, $T_2$ is constructed to be one-to-one.
\item[(b)]  $y^T_n \in \sspan W_{T}$ for each $n\in\nn$.  Let $z^T_n=T_2^{-1}(y^T_n)$ for each $n \inn$ (this is well defined as $T_2$ is one-to-one).
 The sequence $(z_n^T)_{n \inn}$ is a (not normalized) basis for $Z_T$.
\item[(c)] The space $Z_{T} $ is reflexive if and only if $W_{T} $ is weakly compact.
\item[(d)] If $T$ is weakly compact and $(y^T_n)$ shrinking then $W_{T}$ weakly compact.
\end{itemize}
\label{theoremitems}
\end{thm}

We sketch the proof of (b). Note that $y_0 \in W_T$ and $P^T_1(y_0)=\frac{1}{2}y_1$. Hence $y_1 \in \sspan W_T$ because
$P^T_1(W_T) \subset W_T$. Also, for $n >1$, $(P^T_n - P^T_{n-1} )y_0 = \frac{1}{2^{n}}y^T_n \in W_T - W_T$. Thus $y_n \in \sspan W_T$ for each $n \in \nn$.

The next remark follows directly from the definition of the basis (see Theorem \ref{theoremitems}(b)).

\begin{rem}
A sequence $(x_n)_{n \inn}$ in $C(2^\nn)$ is $1$-equivalent to the basis $(z_n^T)_{n \inn}$ of $Z_T$ if and only if for each $(a_n)_n \in c_{00}$
\begin{equation*}
\sum_{m=1}^\infty \| \sum_{n=1}^\infty a_n y^T_n \|_{W^m_T}^2 = \| \sum_{n=1}^\infty a_n x_n\|^2
\end{equation*}
\label{basisequiv}
\end{rem}

\begin{lem} Let $\bbb\subset \llll(X,Y)$ be Borel and suppose the map $\bbb \ni T \mapsto (y_n^T)_{n \inn} \in Y^\nn$
is Borel. Then the following hold:
\begin{itemize}
\item[(a)] The map $\bbb \ni T \mapsto y_0^T \in Y$ is Borel.
\item[(b)] The map $\bbb \ni T \mapsto E_T \in F(Y)$ is Borel.
\item[(c)] The map $\bbb \ni T \mapsto W_T \in F(Y)$ is Borel. Moreover, for each $m \inn$ the map
$\bbb \ni T \mapsto W^m_T \in F(Y)$ is Borel.
\item[(d)] The map $\bbb \times Y \ni (T,y) \mapsto \|y \|_{W^m_T}$ is Borel.
\end{itemize}
\label{manyitems}
\end{lem}

The proof of Lemma \ref{manyitems} will rely on the following
tool from Descriptive Set Theory.

\begin{fact}
Suppose $E$ is a standard Borel space, $P$ is a Polish space and for
each $n \inn$, $f_n: E \to P$ is a Borel map. Then the map $\Phi: E
\to F(P)$ defined by $\Phi(x) = \overline{\{f_n(x) :n \inn\}}$ for
all $x \in E$ is Borel. \label{closuremap}
\end{fact}

\begin{proof}[Proof of Lemma \ref{manyitems}(a)] Let $(x_n)_{n \inn}$ be dense in $B_X$. Define $\tau:\bbb \to Y^\nn$ and $p: Y^\nn \to Y$
by
$$\tau(T) = (y_n^T)_{n \inn} ~\forall T\in\bbb \quad \mbox{ and } \quad p( (x_n)_{n \inn} ) = \sum_{n=1}^\infty \frac{1}{2^n}x_n ~\forall(x_n)\in Y^\nn.$$
By assumption $\tau$ is Borel and it is easy to see that $p$
continuous. Therefore $p\circ \tau$ is Borel. This proves the claim,
as $p\circ\tau (T)=\sum_{n=1}^\infty \frac{1}{2^n}y_n^T=y_0^T$ for
all $T\in\bbb$.
\end{proof}

\begin{proof}[Proof of Lemma \ref{manyitems}(b)] Let $(x_n)_{n \inn}$ be dense in $B_X$ and let $U$ be a non-empty open subset of $Y$. Notice that
\begin{equation*}
\begin{split}
E_T \cap U \not= \emptyset &\iff \exists n \in \mathbb{N},~ q_1,q_2 \in \mathbb{Q}\cap[0,1] \mbox{ with } \\
& \quad \quad  \quad \quad q_1+ q_2 =1,~ q_1(Tx_n)  + q_2(y_0^T) \in U.
\end{split}
\end{equation*}
For $n\in\nn$ and  $q\in\mathbb{Q}\cap[0,1]$ we define
$\tau_{n,q}:\bbb \to Y^2$ and $p: Y^2 \to Y$ by
$$\tau_{n,q}(T) = (qTx_n,(1-q)y_0^T) \mbox{ and } p( (z_1,z_2 )) = z_1+z_2.$$
Using (a) and the definition of strong operator topology, the map
$\tau_{n,q}$ is Borel.  The map $p$ is continuous, and hence
 $p\circ \tau_{n,q}$ is Borel.
The set $\{p\circ \tau_{n,q}(T) :n \inn\textrm{ and }q\in
\mathbb{Q}\cap[0,1]\}$ is dense in $E_T$. Hence, the map $\bbb \ni T
\mapsto E_T \in F(Y)$ is Borel by Fact \ref{closuremap}.
\end{proof}

\begin{proof}[Proof of Lemma \ref{manyitems}(c)]  Let $(x_n)_{n \inn}$ be dense in $B_X$ and let $U$ be a non-empty open subset of $Y$.
For $n,k \inn$ and $ q\in \mathbb{Q}\cap[0,1]$ we define the map
$f_{n,k,q}:\bbb \to Y$ by
$$f_{n,k,q}(T) = P_k(qTx_n + (1-q)y_0^T).$$
Using the same argument used in the proof of Lemma
\ref{manyitems}(b), we have that  $f_{n,k,q}$ is a Borel map. The
set $\{f_{n,k,q}(T) :n,k \inn\textrm{ and }q\in
\mathbb{Q}\cap[0,1]\}$ is dense in $W_T$. Hence, the map $\bbb \ni T
\mapsto W_T \in F(Y)$ is Borel by Fact \ref{closuremap}. The same
argument gives that the map $\bbb \ni T \to W^m_T \in F(Y)$ is Borel
for each $m \inn$.
\end{proof}

\begin{proof}[Proof of Lemma \ref{manyitems}(d)] Let $r \in \mathbb{R}$
with $r>0$ and notice that for $(W,y) \in F(Y) \times Y$
$$\|y \|_{W} < r \iff \exists q \in \qq \mbox{ with } 0<q<r \mbox{ and } y \in qW.$$
Thus, the map $F(Y) \times Y \ni (W,y)\mapsto \|y \|_{W}$ is Borel as $qW$
is closed.  The map $(T,y)\mapsto (W_T^m,y)$ is Borel by part (c).
Hence, the map $(T,y)\mapsto \|y \|_{W^m_T}$ is Borel.
\end{proof}

\begin{lem}
Let $\bbb\subset \llll(X,Y)$ be a Borel set and suppose that the map
$\bbb \ni T \mapsto (y_n^T)_{n \inn} \in Y^\nn$ is Borel.  Then the
following set
$$\fff =\{ (T,(x_n)) \in \bbb \times C(2^\nn)^\nn:  (z_n^T) \mbox{ is 1-equivalent to } (x_n)\}$$
is Borel in $\llll(X,Y) \times  C(2^\nn)^\nn$. \label{lastlemma}
\end{lem}

\begin{proof}

For $k,p,N\inn$ and $a=(a_1,...,a_k)\in\qq^k$, we let
\begin{equation*}
\begin{split}
A_{k,N,a}&=\bigg\{(T,(x_n)) \in \mathcal{F}\, :  \sum_{1 \leq m \leq
N} \bigg\| \sum_{n=1}^k a_n y^T_n \bigg\|_{W^m_T}^2 \leq\bigg\|
\sum_{n=1}^k
a_n x_n \bigg\|^2\bigg\} \mbox{ and }\\
B_{k,p,N,a}&=\bigg\{(T,(x_n)) \in \mathcal{F}\, :  \bigg\|
\sum_{n=1}^k a_n x_n\bigg\|^2 -\frac{1}{p} \leq \sum_{1 \leq m \leq
M } \bigg\| \sum_{n=1}^k a_n y^T_n\bigg\|^2_{W^m_T} \bigg\}.
\end{split}
\end{equation*}
The sets $A_{k,N,a},B_{k,p,N,a}\subset \bbb \times C(2^\nn)^\nn$ are
Borel as the maps $T \mapsto (y_n^T)_{n \inn}$ and $(T,y) \mapsto
\|y \|_{W^m_T}$ are Borel. By Remark \ref{basisequiv}, we have that
$\fff
=\bigcap_{k,N\inn;a\in\qq}A_{k,N,a}\cap\bigcap_{k,p\inn;a\in\qq}\bigcup_{N\inn}B_{k,p,N,a}$,
and hence $\fff$ is Borel.

%
\end{proof}

The next proposition is our main tool for proving Theorems \ref{maintheorem} and \ref{maintheoremSD}.

\begin{prop}
Suppose that $\bbb \subset \llll(X,Y)$ is a Borel collection of weakly compact operators
(resp. operators whose adjoints
have separable range), the map $\bbb
\ni T \mapsto (y_n^T)_{n \inn} \in Y^\nn$ is Borel and for each $T
\in \bbb$ and the space $Z_T$ is reflexive (resp. has separable dual) with basis $(z_n^T)$. Then
there is a reflexive space (resp. space with separable dual) with a basis $Z_\bbb$ such that each
$T\in \bbb$ factors through $Z_\bbb$. \label{general}
\end{prop}

\begin{proof}
We prove the weakly compact case. The case of operators whose adjoints have
separable range is analogous. By Lemma \ref{lastlemma}, the set $$\{ (T,(x_n)) \in \bbb
\times C(2^\nn)^\nn:  (z_n^T) \mbox{ is 1-equivalent to } (x_n)\}$$
is Borel in $\llll(X,Y) \times  C(2^\nn)^\nn$.  Hence, the set
$$\zzz_\bbb=\{ (x_n) \in C(2^\nn)^\nn:\exists
T\in\bbb\textrm{ such that } (z_n^T) \mbox{ is 1-equivalent to }
(x_n)\}$$ is analytic in $C(2^\nn)^\nn$. By Theorem \ref{ADseq1} there
is a reflexive space $Z_\bbb$ such that if $(z_n^T) \in \zzz_\bbb$
the space $Z_T$ is isomorphic to a complemented subspace of $Z_\bbb$.
That is, there exists an embedding $I_T:Z_T\rightarrow Z_\bbb$ and a
bounded projection $P_T:Z_\bbb\rightarrow I_T(Z_T)$. Given the
factorization $T_1:X\rightarrow Z_T$ and $T_2:Z_T \rightarrow Y$
with $T=T_2 T_1$, we now have the factorization $I_T
T_1:X\rightarrow Z_\bbb$ and $T_2 I_T^{-1} P_T:Z_\bbb \rightarrow Y$
with $T=T_2 I_T^{-1}P_T I_T T_1$. Thus, each $T \in \bbb$ factors
through $Z_\bbb$.
\end{proof}

\begin{thm}
Suppose $Y$ has a shrinking basis and $\aaa \subset \llll(X,Y)$
is an analytic collection of weakly compact operators. Then there is a reflexive space with a basis $Z_\aaa$ such that each $T\in \aaa$ factors through $Z_\aaa$.
\end{thm}

\begin{proof}
By Proposition \ref{opscoanalytic} the collection of all weakly compact operators
from $X$ to $Y$ (for separable $X$ and $Y$) is coanalytic.
Using Lusin's separation theorem \cite[Theorem 28.1]{Ke} there is a Borel set $\bbb$ of
weakly compact operators such that $\aaa \subset \bbb$. Let $(y_n)$ be a
shrinking basis for $Y$. For each $T\in \bbb$, set $y_n^T=y_n$ for
each $n \inn$. Clearly, $T \mapsto (y_n^T)_{n \inn}$ is Borel, as it
is constant. Using Theorem \ref{theoremitems}, for each $T \in \bbb$
the space $Z_T$ is reflexive and has a basis $(z_n^T)_{n\inn}$. We
apply Proposition \ref{general} to finish the proof.\end{proof}

Next we prove Theorem \ref{maintheorem} and Theorem
\ref{maintheoremSD} for $Y=C(2^\nn)$.  We will use the method of
slicing and selection developed by Ghoussoub, Maurey and
Schachermayer \cite{GMS}.  This method was used to give alternate
proofs of Zippin's theorems that every reflexive separable  Banach
space embeds into a reflexive Banach space with a basis and every
 Banach space with separable dual embeds into a Banach
space with a shrinking basis.  Dodos and Ferenczi
\cite{DodosFerenczi} showed that it is possible to parametrize
this slicing and selection procedure. We will use their parametrized
selection in our proof.  They proved that given an analytic collection
$\aaa$ of separable reflexive Banach space (respectively Banach
spaces with separable dual), there exists an analytic collection
$\aaa'$ of separable reflexive Banach spaces with bases (respectively
Banach spaces with shrinking bases) such that for all $X\in\aaa$
there exists $Z\in\aaa'$ such that $X$ embeds into $Z$. Before
proceeding to the proof, we must introduce several notions involved
in the slicing and selection procedure.

Let $E$ be a compact metric space.  A map $\Delta: E \times E \to
\rr$ is a {\em fragmentation} if for every closed subset $K$ of $E$
and $\ee>0$ there exists an open subset $V$ of $E$ with $K \cap V
\not=\emptyset$ and such that $\sup\{ \Delta(x,y) : x,y \in K \cap V
\} \leqslant \ee$. Recall that $K(E)$ is the space
of all compact subsets of $E$. In \cite{GMS} they prove the
following.

\begin{thm}[\cite{GMS}]
Let $E$ be a compact metric space and $\Delta$ be a fragmentation on
$E$.  Then there is a function $s_\Delta:K(E) \to E$ called a {\em
dessert selection} satisfying the following:
\begin{itemize}
\item[(i)]  For every non-empty $K \in K(E)$, we have $s_{\Delta}(K) \in K$.
\item[(ii)] If $K \subset C$ are in $K(E)$ and $s_{\Delta}(C) \in K$, then $s_{\Delta}(K)=s_{\Delta}(C)$.
\item[(iii)] If $(K_m)$ are descending in $K(E)$ and $K=\cap_m K_m$, then
$$\lim_m \Delta (s_\Delta (K_m), s_\Delta (K))=0.$$
\end{itemize}
\label{selector}
\end{thm}


\begin{defn}
Let $Z$ be a standard Borel space.  A {\em parametrized Borel fragmentation} on $E$ is a map
$\ddd: Z \times E \times E \to \rr$ such that for each $z \in Z$, setting $\ddd_z(\cdot,\cdot):=\ddd(z,\cdot,\cdot)$
the following are satisfied.
\begin{enumerate}
\item For $z \in Z$, the map $\ddd_z:E \times E \to \rr$ is a fragmentation on $E$.
\item The map $\ddd$ is Borel.
\end{enumerate}
\label{pBf}
\end{defn}

Let $\ddd$ be a parametrized Borel fragmentation on a compact metric
space $E$ with respect to some standard Borel space $Z$.  Define
$s_\ddd:Z \times K(E) \to E$ by $s_\ddd(z,K) = s_{\ddd_z}(K)$ where
$s_{\ddd_z}$ is the dessert selection associated to the
fragmentation $\ddd_z$ and given by Theorem \ref{selector}.  We need
the following important theorem of Dodos.

\begin{thm} \cite[Theorem 5.8]{DodosBook}
Let $E$ be a compact metrizable space and $Z$ be a standard Borel
space.  Let $\ddd: Z \times E \times E \to E$ be a parametrized
Borel fragmentation.  Then the parametrized dessert selection
$s_\ddd:Z \times K(E) \to E$ associated to $\ddd$ is Borel.
\label{pds}
\end{thm}

For convenience, we restate Theorem \ref{maintheoremSD}.

\begin{thm}
Let $X$ be a separable Banach space and let $\aaa \subset \llll(X,
C(2^\nn))$ be a set of bounded operators whose adjoints have
separable range which is analytic in the strong operator topology.
Then there is a Banach space $Z$ with a shrinking basis such that
every $T \in \aaa$ factors through $Z$.
\end{thm}

\begin{proof}
Let $A \subset  \llll(X,C(2^\nn))$ be an analytic collection of operators
whose adjoints have separable range. Using Proposition \ref{opscoanalytic}
the space of all operators whose adjoints have separable range
is coanalytic. Therefore we may apply Lusin's theorem \cite[Lemma 28.1]{Ke} to find
a Borel set $\bbb$ operators whose adjoints have
separable range such that $\aaa \subset \bbb$.

The main step in the proof is to define a parametrized Borel
fragmentation and use the associated parametrized dessert selection
to pick a basis $(y_n^T)_{n \inn}$ of $C(2^\nn)$ such that $T
\mapsto (y_n^T)_{n \inn}$ is Borel and the sequence
$(z_n^T)$ is shrinking. Once this is done we can apply Proposition \ref{general}
to finish the proof.

Define the map $\ddd: \bbb \times 2^\nn \times 2^\nn \to \rr$ by
\begin{equation*}
\ddd(T,\sigma,\tau) = \sup\{|d_n(E_T)(\sigma)- d_n(E_T)(\tau)|: n
\in \nn\}
\end{equation*}
We claim that for each $T \in \bbb$, $\ddd_T=\ddd(T,\cdot,\cdot)$ is
a fragmentation.  To see this, we will follow  the argument in
\cite{GMS}. It will be convenient to define a new operator $T_0:
X\oplus_{1}\ell_1^2\rightarrow C(2^\nn)$ by $T_0(x,a,b)=T(x)+ay_0+b
Id$ for all $(x,a,b)\in X\oplus_{1}\ell_1^2$, where $Id$ denotes the
identity function on $2^\nn$. Note that $T_0^*$ has separable range
because $T^*$ has separable range.
 As $Id\in T_0(B_{X\oplus_1\ell^2_1})$, the
following defines a metric on $C(2^\nn)$,
$$\Delta(\sigma,\tau)=\sup\{|f(\sigma)-f(\tau)|\,:\, f\in T_0(B_{X\oplus_1\ell_1^2})\}\quad\textrm{ for all }\sigma,\tau\in
C(2^\nn).
$$
As $E_T\subseteq  \overline{T_0(B_{X\oplus_1 \ell_1^2})}$, we have
that if $\Delta$ is a fragmentation then
$\ddd_T=\ddd(T,\cdot,\cdot)$ is a fragmentation. For $\sigma\in
2^\nn$, we denote $\delta_\sigma \in C(2^\nn)^*$ to be point
evaluation at $\sigma$. Thus,
$\Delta(\sigma,\tau)=\|T_0^*(\delta_\sigma)-T_0^*(\delta_\tau)\|$.
As $T_0^*$ has separable range, the metric $\Delta$ will be
separable on $2^\nn$. Given $\varepsilon>0$ and $\sigma\in 2^\nn$,
we have that the closed $\varepsilon$-ball about $\sigma$ in the
$\Delta$ metric is given by
$$B_\Delta(\sigma,\varepsilon):=\{\tau\in
2^\nn:\Delta(\sigma,\tau)\leq\varepsilon\}=\cap_{f\in
T_0(B_{X\oplus_1\ell_1^2})}\{\tau\in
2^\nn:|f(\sigma)-f(\tau)|\leq\varepsilon\}.$$ Thus,
$B_\Delta(\sigma,\varepsilon)$ is closed in the usual topology on
$2^\nn$.  Let $\varepsilon>0$ and $K\subseteq 2^\nn$ be closed. We
let $A\subset K$ be a countable subset which is dense in the
$\Delta$ metric.  Thus, $K\subseteq \cup_{\sigma\in A}
B_\Delta(\sigma,\varepsilon/2)$.  By the Baire Category Theorem,
there
 exists $\sigma\in A$ such that $B_\Delta(\sigma,\varepsilon/2)\cap K$ is not relatively nowhere dense.
 Thus, there exists a non-empty open set $V\subseteq B_\Delta(\sigma,\varepsilon/2)\cap
 K$, as $B_\Delta(\sigma,\varepsilon/2)\cap K$ is closed. We have
 that $K\cap V\neq\emptyset$ and $\sup\{ \Delta(x,y) : x,y \in K \cap V
\} \leqslant \varepsilon$.  Thus, $\Delta$ is a fragmentation.

Invoking the Borelness of the maps $(d_n)_{n \inn}$ and the map $T \to E_T$, we have that $\ddd$ is a parametrized Borel
fragmentation according to Definition \ref{pBf}.  By Theorem
\ref{pds} there is a Borel map $s : \bbb \times K(2^\nn) \to 2^\nn$ such that
$s_{T}:K(2^\nn) \to 2^\nn$ defined by $s(T,K) = s_{T}(K)$ is a dessert selection associated to the fragmentation
$\ddd_{T}$.
We will use $s_T$ to select
a basis for $C(2^\nn)$.

Define a sequence $(t_n^T)_{n=0}^\infty$ in $2^{<\nn}$ as follows:  Let $t_0^T = \emptyset$.  Let $\phi:2^{<\nn} \to \nn\cup\{0\}$
denote the unique bijection satisfying $\phi(s) < \phi(t)$ if either $|s| < |t|$, or $|s|=|t|$ and $s<_{\text{lex}} t$. Fix $n \in \nn$ and
$t = \phi^{-1}(n-1)$.  By Theorem \ref{selector} there is a unique $i_t \in \{0,1\}$ such that $t^{\con}i_t \prec s_T(V_t)$, where
$V_s:=\{ \sigma \in 2^\nn: s \prec \sigma\}$ for $s \in 2^{<\nn}$.  Set
\begin{equation}
t_n^T =t^\con j \mbox{ where } j = i_t+1 \mbox{ (mod }2)  \mbox{ and } e_n^T =\chi_{V_{t^T_n}}.
\label{basisdefn}
\end{equation}
As in (see \cite{GMS} and \cite[Claim
5.13 pg. 79]{DodosBook}) $(e^T_n)_{n=0}^\infty$ is a normalized
monotone basis  of $C(2^\nn)$. In order to apply Proposition
\ref{general}, we need the following claim.
\begin{claim}
The map $\bbb \ni T \mapsto (e_n^T)_{n=0}^\infty \in C(2^\nn)^\nn$ is Borel.
\label{Ttobasis}
\end{claim}


\begin{proof}
It is enough to show that for each $n \inn\cup\{0\}$ the map $T
\mapsto e_n^T$ is (call it $\psi$) Borel. Fix $n \inn \cup\{0\}$. If $n=0$ let
$t=\emptyset$; otherwise, let $t=\phi^{-1}(n-1)$. Let
$$B_0 = \{T \in \bbb : t^{\con}1 \prec s(T,V_{t})\} \mbox{ and } B_1= \bbb \setminus B_0.$$
Let
$$f_t: \bbb \to \bbb \times K(2^\nn) \mbox{ be defined by } f_t(T)=(T, V_t).$$
Then $B_0$
and $B_1$ are Borel since $B_0 = f_t^{-1}(s^{-1}(V_{t^{\con}1}))$ and  $B_1 = f_t^{-1}(s^{-1}(V_{t^{\con}0}))$.
By definition
\[ \psi(T) = e_n^T = \left\{ \begin{array}{ll}
               \chi_{V_{t^{\con}0}} & T \in B_0 \\
               \chi_{V_{t^{\con}1}} & T \in B_1.
                \end{array}
        \right. \]
Since $\psi^{-1}(\chi_{V_{t^{\con}0}})=B_0$ and $\psi^{-1}(\chi_{V_{t^{\con}1}})=B_1$, our claim is proved.
\end{proof}

In \cite[Theorem III.1, page 503]{GMS} or \cite[page 80]{DodosBook} they prove $(z_n^T)_{n\inn}$ is a shrinking basis
for $Z_T$. Invoking Proposition \ref{general}, the proof is complete.
\end{proof}

\begin{proof}[Proof of Theorem \ref{maintheorem}]
Now assume that $\aaa \subset \llll(X,C(2^\nn))$ is an analytic collection of weakly compact
operators. This proof follows the same outline as the proof of Theorem \ref{maintheoremSD}. Indeed, it is enough to show that $Z_T$  is reflexive.
Note that we already know $(z_n^T)$ is a shrinking basis for $Z_T$.

By Theorem \ref{theoremitems}(c) it is enough to show that $W_T$ is weakly compact. This is proved in \cite[Lemma 5.18]{DodosBook}. Let $T_2: Z_T \to C(2^\nn)$, be as in Theorem \ref{theoremitems}(b). Set
$K=T_2^{-1}(E_T)$ (note that $T_2^{-1}$ is well defined on $E_T$).  Since $E_T$ is weakly compact, $K$ is a weakly compact subset of $Z_T$. For $k \inn$ let $Q_k:Z_T \to \sspan\{z^T_n: n\leqslant k\}$ be the natural projection. Since $(z_n^T)_{n=1}^\infty$ is shrinking we may use \cite[Lemma 2]{DFJP} (also see \cite[Lemma B.10]{DodosBook}) to conclude that
$$K'= K \cup \bigcup_{k \inn} Q_k(K)$$
is weakly compact. Note that $T_2(K')$ is also weakly compact and
$$T_2(K')= E_T \cup \bigcup_{k \inn} T_2(Q_k(K)) =
E_T \cup \bigcup_{k \inn} P_k(E_T) = \overline{\bigcup_{k \inn} P_k(E_T)} = W_T.$$
This completes the proof.
\end{proof}

\begin{cor}
Suppose $Z$ is a complemented subspace of $C(2^\nn)$ and $\aaa \subset \llll(X,Z)$ be an analytic collection of weakly compact operators (resp. a collection of operators whose adjoints have separable range). Then there is a reflexive space (resp. space with separable dual) $Z_\aaa$ such that each $T \in\aaa$ factors through $Z_\aaa$.
\end{cor}

\section{Analytic collections of spaces}

In this section we present generalizations of Theorems
\ref{maintheorem} and \ref{maintheoremSD}. Our goal is to uniformly
factor sets of operators of the form $T:X\rightarrow Y$, where $X$
and $Y$ are allowed to vary. Our previous results relied on the fact
that both the set of separable Banach spaces and the set of bounded
operators between two fixed separable Banach spaces can be naturally
considered as standard Borel spaces. However, the set of operators
between separable Banach spaces which are allowed to vary is not
immediately realized as a standard Borel space.  To get around this,
we will code operators using sequences.

 Let $X,Y\in SB$ and define $\ccc_{X,Y} \subset C(2^\nn)^\nn$ by
\begin{equation*}
\begin{split}
(w_k)_{n \inn} \in \ccc_{X,Y}  \iff & w_k \in Y, \forall k \inn ~(\forall n,m,l \in \nn,~ \forall q,r \in \qq \\
& d_n(X)= q d_m(X) + p d_l(X) \implies ~ w_n=qw_m+pw_l ) \mbox{ and } \\
& (\exists  K \in \nn, ~\forall (a_i)_i \in \qq^{<\nn}~ \|\sum_i a_i w_i \| \leqslant K \| \sum_i a_i d_i(X)\|).
\end{split}
\end{equation*}
The map defined by $\ccc_{X,Y} \ni (w_k)_{k \inn} \mapsto T \in \llll(X,Y)$, where $T$ is the unique operator $Td_n(X) := w_n$ for each $n\inn$, is an isomorphism. Define $\llll \subset SB \times SB \times C(2^\nn)^\nn$ by
\begin{equation*}
(X,Y,(w_k)) \in \llll  \iff ~(w_k)_{k\inn} \in \ccc_{X,Y}.
\end{equation*}
Note that $\llll$ is a Borel subset of $SB \times SB \times C(2^\nn)^\nn$ and is thus a Standard Borel space.

\begin{prop}
The following subsets of $\llll$ are coanalytic.
\begin{equation*}
\begin{split}
\www = \{(X,Y,(w_k)) \in \llll : & \mbox{ the operator } T \in \llll(X,Y) \mbox{ defined by } \\
&  ~Td_k(X) = w_k \mbox{ for all }k \inn, \mbox{ is weakly compact}\}
\end{split}
\end{equation*}
\begin{equation*}
\begin{split}
\mathcal{SR} = \{(X,Y,(w_k)) \in \llll : & \mbox{ the adjoint of the operator } T \in \llll(X,Y) \mbox{ defined by } \\
&  ~Td_k(X) = w_k \mbox{ for all }k \inn \mbox{ has separable range}\}
\end{split}
\end{equation*}
\label{coanalyticU}
\end{prop}

\begin{proof}
In \cite{BF-ordinal} it is proved that an operator $T:X \to Y$ is weakly compact if for every bounded sequence $(x_n)$ in $B_X$ the image
$(Tx_n)$ does not dominate the summing basis of $c_0$. Let $[\nn]$ denote the set of all infinite increasing sequences in $\nn$. This gives us the following characterization of $\www$
\begin{equation*}
\begin{split}
(X,Y,(w_k)) \in \www \iff & \forall (k_i)_{i \in \nn} \in [\nn],  ~\forall n \in~\exists (a_i) \in \qq^{<\nn}, \\
&\|\sum_{i \in \nn} a_i w_{k_i}\| < \frac{1}{n} \sup_{k \inn} \bigg|\sum_{ i \geqslant k} a_i \bigg|.
\end{split}
\end{equation*}
Therefore $\www$ is coanalytic.

It remains to show that $\mathcal{SR}$ is coanalytic. The proof follows the proof of Proposition \ref{opscoanalytic} after making the following changes to accomodate the triples $(X,Y,(w_k)) \in \mathcal{SR}$.  Let
$$f_{(X,Y,(w_k)),y^*}= \bigg( \frac{y^*(w_n)}{\|d_n(X)\|}\bigg)_{n=1}^\infty \in B_{\ell_\infty}.$$
and
$$K_{(X,Y,(w_k))} = \{ f_{(X,Y,(w_k)),y^*}:  y^* \in B_{Y^*}\}.$$
Finally, define $\mathcal{D} \subset \llll \times B_{\ell_\infty}$ by
$$((X,Y,(w_k)),f) \in \mathcal{D} \iff f \in K_{(X,Y,(w_k))}.$$
As before, $\mathcal{D}$ is Borel and the map $\Phi: \mathcal{L} \to K(B_{\ell_\infty})$ defined  by $\Phi((X,Y,(w_k))= K_{(X,Y,(w_k))}$ is Borel with
$$(X,Y,(w_k)) \in \mathcal{SR} \iff \Phi((X,Y,(w_k)))= \Sigma.$$
Thus, $\mathcal{SR}$ is coanalytic.
\end{proof}

\begin{notation} In this new setting we make set the following notation.  Note that in most cases we are simply replacing $T$ by ${(X,Y,(w_k))}$.
\begin{itemize}
\item[(a)] Let $(X,Y,(w_k)) \in \llll$. Denote by $(y_n^{(X,Y,(w_k))})_{n \in \nn}$ a basis of $Y$ that depends on ${(X,Y,(w_k))}$ and for $k \inn$, let
$P^{(X,Y,(w_k))}_k:Y \to [ y_n^{(X,Y,(w_k))} : n\leqslant k]$ be the natural projection.
\item[(b)] Let $y^{(X,Y,(w_k))}_0=\sum_{n \inn} \frac{1}{2^n} y^{(X,Y,(w_k))}_n$ and
$$E_{(X,Y,(w_k))} :=\overline{co(\{w_k\}_{k \in \nn} \cup\{y^{(X,Y,(w_k))}_0\})}.$$
\item[(c)] Define
\begin{equation*}
W_{(X,Y,(w_k))} = \overline{\bigcup_{k \inn} P^{(X,Y,(w_k))}_k (E_{(X,Y,(w_k))})}.
\end{equation*}
The set $W_{(X,Y,(w_k))}$ is closed, bounded, convex and symmetric. Also, $P^{(X,Y,(w_k))}_k (W_{(X,Y,(w_k))}) \subset W_{(X,Y,(w_k))} $ for each $k \inn$.
\item[(d)] Let
\begin{equation*}
Z_{(X,Y,(w_k))}  = \{ z \in Y : \sum_{m=1}^\infty \|z \|_{W^m_{(X,Y,(w_k))}}^2 < \infty\}
\end{equation*}
\begin{equation*}
 \|z\|_{(X,Y,(w_k))} =  \bigg(\sum_{m=1}^\infty \|z \|_{W^m_{(X,Y,(w_k))}}^2\bigg)^\frac{1}{2}.
\end{equation*}
\end{itemize}
\end{notation}

The next lemma, which we state without proof, are analogous Lemmas \ref{manyitems} and \ref{lastlemma}.

\begin{lem} Let $\bbb\subset \llll$ be Borel and suppose the map $\bbb \ni {(X,Y,(w_k))} \mapsto (y_n^{(X,Y,(w_k))})_{n \inn} \in C(2^\nn)^\nn$
is Borel. Then the following hold:
\begin{itemize}
\item[(a)] The map $\bbb \ni {(X,Y,(w_k))} \mapsto y_0^{(X,Y,(w_k))} \in C(2^\nn)$ is Borel.
\item[(b)] The map $\bbb \ni {(X,Y,(w_k))} \mapsto E_{(X,Y,(w_k))} \in F(C(2^\nn))$ is Borel.
\item[(c)] The map $\bbb \ni {(X,Y,(w_k))} \mapsto W_{(X,Y,(w_k))} \in F(C(2^\nn))$ is Borel. Moreover, for each $m \inn$ the map
$\bbb \ni {(X,Y,(w_k))} \mapsto W^m_{(X,Y,(w_k))} \in F(C(2^\nn))$ is Borel.
\item[(d)] The map $\bbb \times Y \ni ({(X,Y,(w_k))},y) \mapsto \|y \|_{W^m_{(X,Y,(w_k))}}$ is Borel.
\end{itemize}
\label{manyitemsU}
\end{lem}

\begin{lem}
Let $\bbb \subset \llll$ be Borel and $\bbb \ni {(X,Y,(w_k))} \mapsto (y_n^{(X,Y,(w_k))})_{n \inn} \in Y^\nn$
be a Borel map. The set
$$\mathcal{Z} = \{ ((X,Y,(w_k)),E) \in \bbb \times SB : E \mbox{ is isometric to } Z_{(X,Y,(w_k))}\}.$$
is analytic in $\llll \times SB$. \label{Boreluniversal}
\end{lem}

We can now state and prove our main theorem of this section.

\begin{thm}
Set
$$\www_{C(2^\nn)}=\{ (X,Y,(w_k)) \in \www: Y \mbox{ is isomorphic to }C(2^\nn)\}$$
$$\mathcal{SR}_{C(2^\nn)}=\{ (X,Y,(w_k)) \in \mathcal{SR}: Y \mbox{ is isomorphic to }C(2^\nn)\}$$
Suppose that $\aaa$ is an analytic subset of $\www_{C(2^\nn)}$ (resp. $\mathcal{SR}_{C(2^\nn)}$). Then there is
a separable reflexive Banach space with a basis (resp. space with a shrinking basis) $Z$ such that for each $(X,Y,(w_k)) \in \aaa$
the operator $T$ defined by $Td_n(X)=w_n$ for each $n \in \nn$, factors through $Z$.
\end{thm}

\begin{proof}
We will sketch the proof for $\www_{C(2^\nn)}$, the proof in the case of $\mathcal{SR}_{C(2^\nn)}$ is analogous.
Let $\aaa \subset \www_{C(2^\nn)}$ be analytic. Proposition \ref{coanalyticU} and Lusin's theorem \cite[Lemma 18.1]{Ke} together
tell us that $\www_{C(2^\nn)}$ is coanalytic and that there is a Borel subset $\bbb$ of $\www_{C(2^\nn)}$ such that $\aaa \subset \bbb$. The goal is is apply Lemma \ref{Boreluniversal}. Following along the same route we tracked out in the proof of
Theorem \ref{maintheoremSD} we can find for each $(X,Y,(w_k)) \in \www_{C(2^\nn)}$ a basis $(e^{(X,Y,(w_k))}_n)_{n \inn}$ of $C(2^\nn)$ such that the map $(X,Y,(w_k)) \mapsto (e^{(X,Y,(w_k))}_n)_{n \inn}$ is Borel, as desired by Lemma \ref{Boreluniversal}. Again, using the same argument, we claim that for $(X,Y,(w_k)) \in \www_{C(2^\nn)}$ the space $Z_{(X,Y,(w_k))}$ (defined above) is reflexive with a basis. Applying Lemma \ref{Boreluniversal} yields that
$$\mathcal{\mathcal{Z}_\bbb} = \{ Z\in  \bbb : \exists(X,Y,(w_k)) \in \bbb, Z_{(X,Y,(w_k))}=Z\}.$$
is analytic. Therefore, using the same procedure as in the proof of Proposition \ref{general} we obtain a reflexive space $Z_\bbb$ such that every operator $T$ coded by a triple $(X,Y,(w_k)) \in \bbb$ factors through $Z_\bbb$.
\end{proof}

\section{Applications}

In this section we provide several consequences of our uniform factorization results.
In \cite{BF-ordinal} several examples are
given of Banach spaces $X$ and $Y$ such that the space of
weakly compact operators from $X$ to $Y$ is
coanalytic but not analytic. For example, let $U$  be the separable Banach space of Pe{\l}czy{\'n}ski which
contains complemented copies of every Banach space with a basis.
It is shown in \cite{BF-ordinal} that the set of weakly compact operators on $U$
is coanalytic but not Borel.
In terms of factorization, we have the following.

\begin{prop}\label{P27}
There does not exist a separable reflexive space $Z$ such that every
weakly compact operator from $U$ to $C(2^\nn)$ factors through $Z$.
In particular, the set of weakly compact operators from $U$ to
$C(2^\nn)$ is not analytic.
\end{prop}

\begin{proof}
Let $\xi$ be a countable ordinal and let $X_\xi$ be the Tsirelson
space of order $\xi$.  For our purposes we just need that $X_\xi$ is
a reflexive Banach space with a basis and has Szlenk index
$\omega^{\xi\omega}$ \cite{OSZ1}.  We consider $X_\xi$ as a
complemented subspace of $U$ and let $P_\xi:U\to X_\xi$ be a bounded
projection from $U$ onto $X_\xi$. Let  $i_\xi: X_\xi \to C(2^\nn)$
be an embedding of $X_\xi$. The operator $i_\xi$ is weakly compact
as $ X_\xi$ is reflexive, and hence the operator $T_\xi :=
i_\xi\circ P_\xi $ is weakly compact. If there was a Banach space $Z$ with separable dual
such that for all countable ordinals $\xi$ the operator
$T_\xi$ factored through $Z$, then, since $i_\xi$ is an isometry,
$Z$ would contain an isomorphic copy of $X_\xi$ for all countable
ordinals $\xi$.  This would imply that the Szlenk index of $Z$ is
uncountable which contradicts that $Z$ has separable dual \cite{Sz}.
\end{proof}

\begin{prop}\label{P28}
There exists a Banach space $Y$ with a shrinking basis such that
there does not exist a separable reflexive Banach space $Z$ so that
every weakly compact operator on $Y$ factors through $Z$.
In particular, the set of weakly compact operators on $Y$
is not analytic.
\end{prop}

\begin{proof}
Consider the collection $\aaa_{\omega^\omega}$ of all separable Banach spaces with shrinking bases
and Szlenk index less than or equal to $\omega^\omega$. It is shown in \cite{BosPHD} that $\aaa_{\omega^\omega}$ is an analytic subset of $SB$. Using Theorem \ref{ADcomp} there is a Banach space $Y$ with
a shrinking basis such that for each $X$ in $\aaa_{\omega^\omega}$ there is a complemented subspace of $Y$ isomorphic to $X$.  Now let $\xi$ be a countable ordinal and let $X_\xi$ be the Tsirelson
space of order $\xi$.  For this proof, we just need that $X_\xi$ is
a reflexive Banach space with a basis and, Szlenk index
$\omega^{\xi\omega}$ and that $X_\xi^*$ has Szlenk index at most $\omega^\omega$ \cite{OSZ1}.   Thus $X_\xi^*$ is isomorphic to a complemented subspace of $Y$.
 We consider $X_\xi^*$ as a
complemented subspace of $Y$ and let $P_\xi:Y\to X_\xi^*$ be a bounded
projection from $Y$ onto $X_\xi^*$. Let  $i_\xi: X_\xi^* \to Y$
be the identity on $X_\xi^*$. The operator $i_\xi$ is weakly compact
as $X_\xi$ is reflexive, and hence the operator $T_\xi :=
i_\xi\circ P_\xi $ is weakly compact. If there was a reflexive Banach space $Z$
 such that for all countable ordinals $\xi$ the operator
$T_\xi$ factored through $Z$, then, since $i_\xi$ is an isometry,
$Z$ would contain an isomorphic copy of $X_\xi^*$ for all countable
ordinals $\xi$.  Thus, $X_\xi$ would be a quotient of   $Z^*$  for all countable
ordinals $\xi$.  This would imply that the Szlenk index of $Z^*$ is
uncountable which contradicts that $Z$ is reflexive \cite{Sz}.
\end{proof}

In contrast to the negative results of Proposition \ref{P27} and
Proposition \ref{P28}, we have the following theorem.

\begin{thm}\label{T29}
Let $X$ be a Banach space with a shrinking basis such that $X^{**}$
is separable.  The set  of weakly compact operators on $X$ is a
Borel subset of $\llll(X)$.  In particular, there exists a reflexive
Banach space $Z$ such that every weakly compact operator on $X$
factors through $Z$.
\end{thm}
\begin{proof}
Let $(x_k)_{k=1}^\infty$ be a shrinking basis for $X$ with
biorthogonal functionals $(x_k^*)_{k=1}^\infty$, and let $D\subset
X^{**}$ be dense. We denote the set of weakly compact operators on
$X$ by $\www(X)$. By Gantmacher's Theorem, an operator
$T\in\llll(X)$ is weakly compact if and only if
$T^{**}(X^{**})\subseteq X$.  In particular,
\begin{equation}\label{Eapp1}
T\in\www(X)\Leftrightarrow T^{**}f\in X\,\forall f\in D.
\end{equation}
  Since
$(x_k)_{k=1}^\infty$ is a $w^*-$basis for $X^{**}$,
\begin{equation}\label{Eapp2}
f=w^*-\lim_{n\rightarrow\infty}\sum_{i=1}^n x_i^*(f)x_i\quad\textrm{
for all } f\in X^{**}.
\end{equation}
Thus, we have for all $f\in X^{**}$ that
\begin{equation}\label{Eapp3}
T^{**}f\in X \Leftrightarrow T^{**}f=
\|\cdot\|-\lim_{n\rightarrow\infty} \sum_{k=1}^n x_k^*(T^{**}f)x_k
\Leftrightarrow \lim_{M\rightarrow\infty}\lim_{N\rightarrow\infty}
\left\|\sum_{k=M}^N x_k^*(T^{**}f)x_k\right\|=0
\end{equation}

Note that $T^{**}(x)=T(x)$ for all $x\in X$.  Since $T^{**}$ is
$w^*$ to $w^*$ continuous,
$$T^{**}f=w^*-\lim_{n\rightarrow\infty}\sum_{i=1}^n x_i^*(f)T(x_i)\quad\textrm{ for
all }f\in X^{**}.
$$
Hence, for all $k\in\nn$, we have  that
$$x^*_k(T^{**}f)x_k=\lim_{n\rightarrow\infty}x^*_k\left(\sum_{i=1}^n x_i^*(f)T(x_i)\right)x_k.$$
Substituting into (\ref{Eapp3}) gives,
$$ T^{**}f\in X
\Leftrightarrow \lim_{M\rightarrow\infty}\lim_{N\rightarrow\infty}
\lim_{n\rightarrow\infty}\left\|\sum_{k=M}^N x^*_k\left(\sum_{i=1}^n
x_i^*(f)T(x_i)\right)x_k\right\|=0
$$

Thus, $\www(X)$ is Borel by \ref{Eapp1}.

\end{proof}

Let $J$ be the quasi-reflexive space of James \cite{JamesSpacePNAS}.
Laustsen \cite{Laustsen} proved the following result by constructing
the required space. As $J$ has a shrinking basis and $J^{**}$ is
separable, we obtain it as a corollary of Theorem \ref{T29}.

\begin{prop}
There is a reflexive space $Z$ such that every weakly compact
operator on $J$ factors through $Z$.
\end{prop}


In \cite{Lindenstrauss}, Lindenstrauss showed that for each separable Banach space
$X$ there is a separable Banach space $Y$ such that $Y^{**}/Y$ is isomorphic to $X$. In particular, the space $Y$ has separable bidual. Therefore, Theorem \ref{T29} yields that whenever $Y$ has a basis, every weakly compact operator on $Y$ factors through a single reflexive space.

\begin{prop}
Let $X$ and $Y$ be separable Banach spaces. Then every closed norm-separable set $S$ of weakly compact operators is Borel in the strong operator topology.
\label{cpt}
\end{prop}

\begin{proof}
 Let $(T_k)_{k \inn}$ be a dense subset of $S$ and $(d_k)_{k \inn}$ be dense in $B_X$.
Then
\begin{equation*}
\begin{split}
T \in S
 & \iff \forall m\inn,~ \exists k \in \nn \mbox{ such that } \forall j \inn,~ \|(T - T_k)d_j \| < \frac{1}{m}.
\end{split}
\end{equation*}
From this characterization it follows that $S$ is Borel.
\end{proof}

One corollary of Proposition \ref{cpt} is that if $X^*$ or $Y$
has the approximation property, then the set of compact operator from $X$ to
$Y$ is Borel.
In \cite{JohnsonCompact} Johnson proved that there is a space $Z_K$
such that every operator which is the uniform limit of finite rank
operators (independent of the spaces $X$ and $Y$) factors through
$Z_K$. In particular this implies that whenever either $X^*$ or $Y$
has the approximation property every compact operator from $X$ to
$Y$ factors through $Z_K$. Johnson and Szankowski
\cite{JS} proved that there is no separable Banach space such that
every compact operator factors through it. The following result follows from
Proposition \ref{cpt} and Theorem \ref{maintheorem} and is a weaker version of Johnson's Theorem.

\begin{cor}
If  $Y$ is Banach space with a shrinking basis or is isomorphic to $C(2^\nn)$
then there exists a reflexive space $Z$ such that if $X$ is a separable Banach space with the approximation property then every compact operator from $X$ to $Y$ factors through $Z$.
\end{cor}

\begin{proof}
Let  $Y$ either be a Banach space with a shrinking basis or be isomorphic to $C(2^\nn)$.
By Proposition \ref{cpt} and Theorem \ref{maintheorem}, there exists a reflexive Banach space $Z$ such that every compact operator from $U$ to $Y$ factors through $Z$.  If $X$ is a separable Banach space with the approximation property then $X$ is isomorphic to a complemented subspace of $U$.  Every compact operator from $X$ to $Y$ has a compact factorization through $U$ and hence factors through $Z$ as well.
\end{proof}

\begin{prop}
There exists a separable hereditarily indecomposable Banach space $X$, with HI dual and non-separable bidual, and a reflexive Banach space $Z$ such that every weakly compact operator on $X$ factors through $Z$.
\end{prop}

\begin{proof}
In \cite{AAT} the authors construct an HI space $X$ with a shrinking basis such that $X^*$ is HI and  $X^{**}$ is non-separable  and on which every operator is a scalar multiple of the identity plus a weakly compact operator. Once again it suffices to show that the set of weakly compact operators on $X$ is Borel. In \cite{AAT} they prove that each weakly compact operator on $X$ is strictly singular. It is shown in \cite{BeanlandIsrael} that when the strictly singular operators have codimension-one in $\llll(X)$ they are a Borel subset. It follows that the set of weakly compact operators on $X$ is a Borel subset of $\llll(X)$.  Hence, we may apply Theorem \ref{maintheorem}.
\end{proof}


\bibliographystyle{abbrv}
\bibliography{bib_source}

\def\cprime{$'$} \def\cprime{$'$} \def\cprime{$'$}
\begin{thebibliography}{10}

\bibitem{AAT}
S.~A. Argyros, A.~D. Arvanitakis, and A.~G. Tolias.
\newblock Saturated extensions, the attractors method and hereditarily {J}ames
  tree spaces.
\newblock In {\em Methods in {B}anach space theory}, volume 337 of {\em London
  Math. Soc. Lecture Note Ser.}, pages 1--90. Cambridge Univ. Press, Cambridge,
  2006.

\bibitem{ADo}
S.~A. Argyros and P.~Dodos.
\newblock Genericity and amalgamation of classes of {B}anach spaces.
\newblock {\em Adv. Math.}, 209(2):666--748, 2007.

\bibitem{ALRR}
R.~Aron, M.~Lindstr{\"o}m, W.~M. Ruess, and R.~Ryan.
\newblock Uniform factorization for compact sets of operators.
\newblock {\em Proc. Amer. Math. Soc.}, 127(4):1119--1125, 1999.

\bibitem{BeanlandIsrael}
K.~Beanland.
\newblock An ordinal indexing on the space of strictly singular operators.
\newblock {\em Israel J. Math.}, 182:47--59, 2011.

\bibitem{BeanlandDodos}
K.~Beanland and P.~Dodos.
\newblock On strictly singular operators between separable {B}anach spaces.
\newblock {\em Mathematika}, 56(2):285--304, 2010.

\bibitem{BF-ordinal}
K.~Beanland and D.~Freeman.
\newblock Ordinal ranks on weakly compact and {R}osenthal operators.
\newblock {\em Extracta Math.}, 26(2):173--194, 2011.

\bibitem{BosPHD}
B.~Bossard.
\newblock {\em Th\'eorie descriptive des ensembles en g\'eom\'etrie des espaces
  de Banach}.
\newblock PhD thesis, 1994.

\bibitem{Bos}
B.~Bossard.
\newblock A coding of separable {B}anach spaces. {A}nalytic and coanalytic
  families of {B}anach spaces.
\newblock {\em Fund. Math.}, 172(2):117--152, 2002.

\bibitem{Bo}
J.~Bourgain.
\newblock On separable {B}anach spaces, universal for all separable reflexive
  spaces.
\newblock {\em Proc. Amer. Math. Soc.}, 79(2):241--246, 1980.

\bibitem{BrookerAsplund}
P.~A.~H. Brooker.
\newblock Asplund operators and the {S}zlenk index.
\newblock {\em J.Operator Theory}, 68:405--442, 2012.

\bibitem{DFJP}
W.~J. Davis, T.~Figiel, W.~B. Johnson, and A.~Pe{\l}czy{\'n}ski.
\newblock Factoring weakly compact operators.
\newblock {\em J. Functional Analysis}, 17:311--327, 1974.

\bibitem{DodosBook}
P.~Dodos.
\newblock {\em Banach spaces and descriptive set theory: selected topics},
  volume 1993 of {\em Lecture Notes in Mathematics}.
\newblock Springer-Verlag, Berlin, 2010.

\bibitem{DodosQ}
P.~Dodos.
\newblock Quotients of banach spaces and surjectively universal spaces.
\newblock {\em Studia Math.}, 197:171--194, 2010.

\bibitem{DodosFerenczi}
P.~Dodos and V.~Ferenczi.
\newblock Some strongly bounded classes of {B}anach spaces.
\newblock {\em Fund. Math.}, 193(2):171--179, 2007.

\bibitem{Fi}
T.~Figiel.
\newblock Factorization of compact operators and applications to the
  approximation problem.
\newblock {\em Studia Math.}, 45:191--210. (errata insert), 1973.

\bibitem{FOSZ}
D.~Freeman, E.~Odell, T.~Schlumprecht, and A.~Zs{\'a}k.
\newblock Banach spaces of bounded {S}zlenk index. {II}.
\newblock {\em Fund. Math.}, 205(2):161--177, 2009.

\bibitem{GMS}
N.~Ghoussoub, B.~Maurey, and W.~Schachermayer.
\newblock Slicings, selections and their applications.
\newblock {\em Canad. J. Math.}, 44(3):483--504, 1992.

\bibitem{GoGu}
M.~Gonz{\'a}lez and J.~M. Guti{\'e}rrez.
\newblock Factoring compact sets of operators.
\newblock {\em J. Math. Anal. Appl.}, 255(2):510--518, 2001.

\bibitem{JamesSpacePNAS}
R.~C. James.
\newblock A non-reflexive {B}anach space isometric with its second conjugate
  space.
\newblock {\em Proc. Nat. Acad. Sci. U. S. A.}, 37:174--177, 1951.

\bibitem{JohnsonCompact}
W.~B. Johnson.
\newblock Factoring compact operators.
\newblock {\em Israel J. Math.}, 9:337--345, 1971.

\bibitem{JS}
W.~B. Johnson and A.~Szankowski.
\newblock Complementably universal {B}anach spaces. {II}.
\newblock {\em J. Funct. Anal.}, 257(11):3395--3408, 2009.

\bibitem{Ke}
A.~S. Kechris.
\newblock {\em Classical descriptive set theory}, volume 156 of {\em Graduate
  Texts in Mathematics}.
\newblock Springer-Verlag, New York, 1995.

\bibitem{KRN}
K.~Kuratowski and C.~Ryll-Nardzewski.
\newblock A general theorem on selectors.
\newblock {\em Bull. Acad. Polon. Sci. S\'er. Sci. Math. Astronom. Phys.},
  13:397--403, 1965.

\bibitem{Laustsen}
N.~J. Laustsen.
\newblock Maximal ideals in the algebra of operators on certain {B}anach
  spaces.
\newblock {\em Proc. Edinb. Math. Soc. (2)}, 45(3):523--546, 2002.

\bibitem{Lindenstrauss}
J.~Lindenstrauss.
\newblock On {J}ames's paper ``{S}eparable conjugate spaces''.
\newblock {\em Israel J. Math.}, 9:279--284, 1971.

\bibitem{MO}
K.~Mikkor and E.~Oja.
\newblock Uniform factorization for compact sets of weakly compact operators.
\newblock {\em Studia Math.}, 174(1):85--97, 2006.

\bibitem{OSuconvex}
E.~Odell and T.~Schlumprecht.
\newblock A universal reflexive space for the class of uniformly convex banach
  spaces.
\newblock {\em Math. Ann.}, 335(4):901--916, 2006.

\bibitem{OSZ1}
E.~Odell, T.~Schlumprecht, and A.~Zs{\'a}k.
\newblock Banach spaces of bounded {S}zlenk index.
\newblock {\em Studia Math.}, 183(1):63--97, 2007.

\bibitem{Sz}
W.~Szlenk.
\newblock The non-existence of a separable reflexive {B}anach space universal
  for all separable reflexive {B}anach spaces.
\newblock {\em Studia Math.}, 30:53--61, 1968.

\bibitem{Z}
M.~Zippin.
\newblock Banach spaces with separable duals.
\newblock {\em Trans. Amer. Math. Soc.}, 310(1):371--379, 1988.

\end{thebibliography}

\end{document}